\documentclass[12pt]{amsart}

\usepackage{amsmath,amssymb,amsfonts,amsthm}
\usepackage{mathtools}
\usepackage[margin=1in]{geometry}

\usepackage{graphicx}
\usepackage{tikz}
\usetikzlibrary{arrows.meta,graphs,shapes.geometric,decorations.pathreplacing}
\usepackage{pgfplots}
\pgfplotsset{compat=1.18}

\usepackage{booktabs}
\usepackage{multicol}
\usepackage{enumitem}

\usepackage[numbers,sort&compress]{natbib}

\usepackage{xcolor}
\definecolor{myteal}{RGB}{0,128,128}
\usepackage[pdfborder={0 0 0}]{hyperref}
\hypersetup{
  colorlinks=true,
  citecolor=myteal,
  linkcolor=blue,
  urlcolor=myteal
}

\usepackage{aliascnt}
\usepackage[capitalise, noabbrev]{cleveref}

\usepackage{parskip}

\usepackage{listings}
\newcommand{\newaliastheorem}[3]{%
  \newaliascnt{#1}{thm}%
  \newtheorem{#1}[#1]{#2}%
  \aliascntresetthe{#1}%
  \crefname{#1}{#2}{#3}%
  \Crefname{#1}{#2}{#3}%
}

\theoremstyle{plain}
\newtheorem{thm}{Theorem}[section]
\crefname{thm}{Theorem}{Theorems}
\Crefname{thm}{Theorem}{Theorems}

\newaliastheorem{lem}{Lemma}{Lemmas}
\newaliastheorem{prop}{Proposition}{Propositions}
\newaliastheorem{cor}{Corollary}{Corollaries}
\newaliastheorem{clm}{Claim}{Claims}
\newaliastheorem{conj}{Conjecture}{Conjectures}

\theoremstyle{definition}
\newaliastheorem{defn}{Definition}{Definitions}
\newaliastheorem{ex}{Example}{Examples}
\newaliastheorem{notation}{Notation}{Notations}
\newaliastheorem{setup}{Set-up}{Set-ups}
\newaliastheorem{remark}{Remark}{Remarks}
\newaliastheorem{problem}{Problem}{Problems}
\newaliastheorem{question}{Question}{Questions}
\newaliastheorem{observation}{Observation}{Observation}

\setlist[enumerate]{leftmargin=*,itemsep=2pt,topsep=4pt}
\SetEnumerateShortLabel{a}{\textup{(\alph*)}}
\SetEnumerateShortLabel{A}{\textup{(\Alph*)}}
\SetEnumerateShortLabel{1}{\textup{(\arabic*)}}
\SetEnumerateShortLabel{i}{\textup{(\roman*)}}
\SetEnumerateShortLabel{I}{\textup{(\Roman*)}}

\makeatletter
\newtheorem*{rep@theorem}{\rep@title}
\newcommand{\newreptheorem}[2]{%
\newenvironment{rep#1}[1]{%
 \def\rep@title{#2 \ref{##1}}%
 \begin{rep@theorem}}%
 {\end{rep@theorem}}}
\makeatother

\newreptheorem{theorem}{Theorem}

\DeclareMathOperator{\Ind}{Ind}

\begin{document}

\title{Independence Polynomials of Graphs}

\author[Hibi]{Takayuki Hibi}
\address[T.~Hibi]{Department of Pure and Applied Mathematics, Graduate School of Information Science and Technology, Osaka University, Suita, Osaka 565--0871, Japan}
\email{\href{mailto:hibi@math.sci.osaka-u.ac.jp}{hibi@math.sci.osaka-u.ac.jp}}

\author[Kara]{Selvi Kara}
\address[S.~Kara]{Department of Mathematics, Bryn Mawr College, Bryn Mawr, PA 19010}
\email{\href{mailto:skara@brynmawr.edu}{skara@brynmawr.edu}}

\author[Vien]{Dalena Vien}
\address[D.~Vien]{Department of Mathematics, Bryn Mawr College, Bryn Mawr, PA 19010}
\email{\href{mailto:dvien@brynmawr.edu}{dvien@brynmawr.edu}}

	

    	   	
	\thanks{The present paper was completed while the first author stayed at Bryn Mawr College, Pennsylvania, February 28 -- March 21, 2026. } 

\begin{abstract}
In this paper, we study the independence polynomial \(P_G(x)\) of a finite simple graph \(G\), with emphasis on the evaluation at \(x=-1\), symmetry, and its connection with the \(h\)-polynomial of the edge ideal of $G$. For big star graphs, we determine exactly when \(P_G(-1)\) is \(0\), \(1\), or \(-1\), characterize the pseudo-Gorenstein\(^*\) members, and show that there is a unique big star with symmetric independence polynomial. We also study graphs obtained from a graph \(H\) by attaching leaves to selected vertices. We derive an explicit formula for the resulting independence polynomial, determine the corresponding value at \(-1\), and prove that if every vertex of \(H\) receives at least one leaf, then the independence polynomial is symmetric if and only if each vertex receives exactly two leaves. As an application, we obtain exact criteria for the values of \(P_G(-1)\) and for the pseudo-Gorenstein\(^*\) members of caterpillar graphs. For cochordal graphs, we classify all symmetric independence polynomials. Finally, for connected  graphs on $n$ vertices with small independence numbers, we determine the exact range of possible values of \(P_G(-1)\).
\end{abstract}

\maketitle

\section*{Introduction}

For a finite simple graph \(G\), the \emph{independence polynomial} of \(G\) is
\[
P_G(x)=\sum_{i=0}^{\alpha(G)} g_i(G)x^i,
\]
where \(g_i(G)\) denotes the number of independent sets of \(G\) of size \(i\).
This polynomial is one of the basic enumerative invariants attached to a graph.  Beyond its purely combinatorial meaning, the independence polynomial appears as the partition function of the hard-core model in statistical physics \cite{hardcore_model} and
its zeros play an important role in questions
related to the Lov\'asz local lemma and zero-free regions from probabilistic combinatorics \cite{ScottSokal2005}. Thus independence
polynomials sit at the intersection of graph enumeration, statistical physics,
probabilistic combinatorics, and commutative algebra; see, for example,
\cite{Survey_ind_poly,GutmanHarary,biermann2026,hibi2026pseudogorensteingraphs}.

In this paper, we focus on two aspects of \(P_G(x)\): the evaluation at \(x=-1\) and
symmetry.  The value  $P_G(-1)$
is the alternating count of independent sets, equivalently the negative of the reduced
Euler characteristic of the independence complex \(\Ind(G)\). This quantity has been
studied from several points of view, including bounds in terms of graph structure,
decycling number, and realizability problems; see
\cite{Upper_bound_engstrom,levit2009independence_preprint,Levit_Mandrescu,Bound_indp_poly_neg_1}. The second theme, symmetry of the independence polynomial, belongs to the broader study
of  palindromicity, unimodality, log-concavity, and roots. There is a substantial literature on constructions that force symmetry, notably the two-whisker
construction of Stevanović and its later extensions and refinements; see
\cite{Stevanovic,LevitMandrescuSymmetry,Mandrescu2012}. Even for trees, these questions
remain subtle: for example, if \(T\) is a tree, then \(P_T(-1)\in\{-1,0,1\}\) by
\cite{levit2009independence_preprint}, yet it is far from clear which trees realize
which value.

A second source of motivation to study independence polynomials comes from commutative algebra. Let \(I(G)\) be the
edge ideal of \(G\), and let \(h_G(t)\) denote the \(h\)-polynomial of  \(S/I(G)\). It was discussed by the authors in \cite{hibi2026pseudogorensteingraphs} that the value \(P_G(-1)\) controls the top
coefficient of \(h_G(t)\). Moreover,  the multiplicity of \(-1\) as a root of \(P_G(x)\), denoted by $M(G)$,
controls the degree of the \(h\)-polynomial by \cite[Theorem~4.4]{biermann2026},
\[
\deg h_G(t)=\alpha(G)-M(G).
\]
Thus the multiplicity of \(-1\) as a root of \(P_G(x)\) determines the
\(\mathfrak a\)-invariant of \(S/I(G)\). Additionally, the condition that \(G\)
be pseudo-Gorenstein\(^*\), a new concept introduced and studied by the authors in \cite{hibi2026pseudogorensteingraphs}, is equivalent to the simple identity
\[
P_G(-1)=(-1)^{\alpha(G)}.
\]
This makes the evaluation at \(-1\) a natural meeting point between the
combinatorics of independent sets and the algebra of edge ideals.

Our first focus  is on \emph{big star graphs}, obtained by gluing several
paths at a common center. For these graphs, we show that the vanishing of
\(P_G(-1)\) is completely determined by congruence classes modulo \(3\), while the
sign in the nonzero cases is controlled by congruence classes modulo \(6\); see
\Cref{thm:big_star_0,thm:big_star_pm1}. We then obtain a complete
description of pseudo-Gorenstein\(^*\) big stars in
\Cref{cor:pseudogorenstein_star_big_star}. We also study symmetry and prove in
\Cref{prop:big_star_symmetry} that, among all big stars, there is exactly one with
symmetric independence polynomial.

Our second focus  is on the effects of \emph{whiskering}, attaching leaves. In
\Cref{prop:leaf-attachment} we provide a general formula for the independence
polynomial of a graph obtained from a base graph \(H\) by attaching \(f_i\) leaves
to selected vertices \(x_i\). When the polynomial is evaluated at \(-1\), a
striking simplification occurs: only the set \(C\subseteq V(H)\) of vertices that
actually receive leaves matters. More precisely, \(P_G(-1)=0\) unless \(C\) is an
independent set of \(H\), and when \(C\) is independent we obtain
\[
P_G(-1)=(-1)^{|C|}P_{H-N_H[C]}(-1).
\]
We also prove in \Cref{thm:full_leaf_attachment_symmetric} that if every vertex of
\(H\) receives at least one leaf, then the resulting graph has symmetric
independence polynomial if and only if exactly two leaves are attached to each
vertex. This recovers Stevanović's two-whisker construction \cite{Stevanovic},
which was later shown to yield not only symmetry but also unimodality
\cite{Mandrescu2012}, and shows that in our leaf-attachment setting it is in fact
the unique symmetric case.

As an application of the whiskering, we study \emph{caterpillar
graphs}, obtained by attaching leaves to paths. Encoding the locations of the legs along the path leads to a
product formula for \(P_G(-1)\) in \Cref{thm:caterpillar-minus-one}. From this we
determine when  \(0,\pm1\) occurs; see
\Cref{cor:caterpillar-zero,cor:caterpillar_pm1_sign}. In addition, we 
characterize pseudo-Gorenstein\(^*\) caterpillars in
\Cref{cor:pseudogorenstein_star_caterpillar}.

We also study \emph{cochordal graphs}, equivalently graphs whose complements are
chordal. This class is natural both combinatorially and algebraically: independent
sets of \(G\) are cliques of \(\overline{G}\), and cochordal graphs play a central
role in the study of edge ideals with linear resolutions \cite{Froberg,HHZ_linear_res}. Since clique complexes of
chordal graphs are quasi-forests \cite{HHZ-CM-chordal}, the independence polynomial
of a cochordal graph can be approached through the \(b\)-sequence description of
quasi-forests from \cite{Combinatorica_cochordal}.  Using \cite{Combinatorica_cochordal}, we prove that a cochordal graph
\(G\) with \(d=\alpha(G)\geq 2\) has symmetric independence polynomial if and only
if
\[
P_G(x)=(1+x)^d+mx(1+x)^{d-2}
\]
for some \(m\geq 0\). In particular, every symmetric independence polynomial of a
cochordal graph is automatically unimodal. We also recall that $P_G(-1)=1-k$
where $k$ is the number of connected components of \(\overline{G}\); see
\Cref{cor:cochordal-k-components-minus-one}.

Finally, we investigate how large or small \(P_G(-1)\) can be when the
independence number is small. For connected graphs on \(n\) vertices with
\(\alpha(G)\le 2\), \Cref{thm:alpha2_range} shows that every integer in the
interval
\[
\left[-(n-1),  \left\lfloor \frac{(n-2)^2}{4}\right\rfloor-1\right]
\]
occurs as \(P_G(-1)\). The upper bound comes from Mantel's theorem applied to the
complement, while the full interval is realized by an explicit connected chordal
family built from two cliques. In this way, our final section may be viewed as a counterpart to earlier realizability results for
\(P_G(-1)\), such as the connected constructions of Cutler and Kahl from \cite{Bound_indp_poly_neg_1}.

\section{Preliminaries}

Throughout the paper, all graphs are finite and simple. For a graph \(G\), we write
\(V(G)\) and \(E(G)\) for its vertex and edge sets, \(\alpha(G)\) for its
independence number, and \(\overline{G}\) for its complement. We denote by
\(\Ind(G)\) the independence complex of \(G\), i.e. the simplicial complex whose
faces are the independent subsets of \(V(G)\). For \(W\subseteq V(G)\), let
\(N_G(W)\) and \(N_G[W]=W\cup N_G(W)\) denote the open and closed neighborhoods of
\(W\), and let \(G-W\) be the induced subgraph on \(V(G)\setminus W\). In
particular, \(G-v:=G-\{v\}\) and \(G-N[v]:=G-N_G[v]\). We write \(G\sqcup H\) for
the disjoint union of graphs \(G\) and \(H\). For \(n\ge 0\), let \(P_n\) denote
the path on \(n\) vertices, with the convention that \(P_0\) is the empty graph.

The independence polynomial of \(G\) is
\[
P_G(x)=\sum_{i=0}^{\alpha(G)} g_i(G)x^i
      =\sum_{F\in \Ind(G)} x^{|F|},
\]
where \(g_i(G)\) is the number of independent sets of \(G\) of size \(i\).
When the graph is clear from context, we simply write \(g_i\).

We say that \(P_G(x)\) is \emph{symmetric} if
\[
x^{\alpha(G)}P_G(1/x)=P_G(x),
\]
equivalently, if \(g_i(G)=g_{\alpha(G)-i}(G)\) for all \(i\).
We repeatedly use the multiplicativity
\[
P_{G\sqcup H}(x)=P_G(x)P_H(x).
\]

Let \(G\) have vertex set \(\{x_1,\dots,x_n\}\), let
\(S=K[x_1,\dots,x_n]\), and let
\[
I(G)=(x_ix_j:\{x_i,x_j\}\in E(G))
\]
be the edge ideal of \(G\).

Recall that $h_G(t)$ denotes the $h$-polynomial of $S/I(G)$ and 
\[
h_G(t):=h_{S/I(G)}(t)=\sum_{i=0}^{\alpha(G)} h_i(G)t^i,
\]
where \(h_i(G)=0\) for \(i>\deg h_G(t)\).
We denote the $\mathfrak a$-invariant of $S/I(G)$ by $\mathfrak a(G)$ where
\[
\mathfrak a(G):=\mathfrak a(S/I(G))=\deg h_G(t)-\alpha(G).
\]
Lastly, we use $M(G)$ to denote the multiplicity of $x=-1$ as a root of the independence polynomial $P_G(x)$.

We now recall a collection of standard results and consequences (see \cite{hibi2026pseudogorensteingraphs}).
\begin{prop}\label{prop:h-from-independence}
Let \(\alpha=\alpha(G)\). Then
\[
h_G(t)=(1-t)^{\alpha}P_G\!\left(\frac{t}{1-t}\right).
\]
In particular, \(h_{\alpha}(G)=(-1)^{\alpha}P_G(-1)\).
\end{prop}

The next theorem, proved in \cite[Theorem 4.4]{biermann2026}, relates the degree
of \(h_G(t)\) to the multiplicity of \(-1\) as a root of \(P_G(x)\).

\begin{thm}\cite[Theorem 4.4]{biermann2026}\label{thm:deg-via-ord}
For every finite simple graph \(G\),
\[
\deg h_G(t)=\alpha(G)-M(G).
\]
\end{thm}

Thus, it follows from the above result that   $\mathfrak a(G)=-M(G)$ for every finite simple graph \(G\).  Moreover, $\mathfrak a(G)=0$ if and only if $P_G(-1)\neq 0$.

Next, we recall the definition of pseudo-Gorenstein\(^{*}\) graphs from \cite{hibi2026pseudogorensteingraphs}.

\begin{defn}
Let \(G\) be a finite simple graph and let \(\alpha=\alpha(G)\).
We say that \(G\) is \emph{pseudo-Gorenstein\(^{*}\)} if $\mathfrak a(G)=0$ and $h_{\alpha}(G)=1$.
\end{defn}

It then follows that a graph \(G\) is pseudo-Gorenstein\(^{*}\) if and only if $P_G(-1)=(-1)^{\alpha(G)}$.

The following standard vertex-deletion recursion for independence polynomials is extremely useful for computations.

\begin{lem}\label{lem:delete_contract}
Let \(G\) be a finite simple graph and let \(v\in V(G)\). Then
\[
P_G(x)=P_{G-v}(x)+xP_{G-N[v]}(x).
\]
\end{lem}

For later reference, we record the following general fact about the value of the independence polynomial of a tree at $x=-1$.

\begin{thm}\cite[Theorem 2.4]{levit2009independence_preprint}\label{levit:tree_values}
If \(T\) is a tree, then $P_T(-1)\in\{-1,0,1\}$.
\end{thm}

The values of \(P_{P_n}(-1)\) for paths have been studied in the literature; see, for example, \cite{Kara_Vien_2026}.

\begin{lem}\label{lem:path-minus-one}
Let \(p_n:=P_{P_n}(-1)\) for \(n\ge 0\). Then
\[
p_0=1,\qquad p_1=0,\qquad p_n=p_{n-1}-p_{n-2}\quad \text{for } n\ge 2.
\]
Consequently, \(p_n\) is periodic of period \(6\), and
\[
p_n=
\begin{cases}
1,& n\equiv 0,5 \pmod 6,\\[4pt]
0,& n\equiv 1,4 \pmod 6,\\[4pt]
-1,& n\equiv 2,3 \pmod 6.
\end{cases}
\]
\end{lem}

\section{Big Stars}

In this section we study the independence polynomial of big star graphs, obtained by gluing several paths at a common center. This family provides a natural testing ground in which the interaction between local path structure and global branching can still be analyzed explicitly. Using the vertex-deletion recursion at the center, together with the periodic behavior of \(P_{P_n}(-1)\), we determine exactly when the value of the independence polynomial at \(-1\) is \(0\), \(1\), or \(-1\). We then characterize pseudo-Gorenstein\(^*\) big stars. Finally, we turn to symmetry and show that there is only one big star whose independence polynomial is symmetric.

\begin{defn}
    Let $q\geq 3$ be an integer and $n_1,\ldots, n_q$ be positive integers. Let $\Gamma_i$ denote the path on the vertex set $V_i=\{x,x_1^{(i)}, \ldots, x_{n_i}^{(i)}\}$ for $1\leq i\leq q$. $\Gamma_i$ is a path of length $n_i$ for each $i$. We assume $V_i\cap V_j = \{x\}$ for $i\neq j$. Let $G(n_1,\ldots, n_q)=\Gamma_1\cup \cdots \cup \Gamma_q$. Using the terminology of \cite[Definition 5.9]{bhaskara2025levelable} we call $G(n_1,\ldots, n_q)$ a big star graph with a center vertex $x$.

    For the rest of this section, let  \(r:=|\{ i : n_i \text{ is odd} \}|\) and let $x$ be the center vertex of $G(n_1,\ldots, n_q)$.
\end{defn}

\begin{lem}\label{lem:big_star}
Let $G(n_1,\ldots,n_q)$ be a big star graph. Then 
\[P_G(x)=\prod_{i=1}^q P_{P_{n_i}}(x)+x\prod_{i=1}^q P_{P_{n_i-1}}(x).\]
\end{lem}

\begin{proof}
Write $G:=G(n_1,\ldots,n_q)$ and let $x$ be its center vertex.  By deleting the center vertex $x$, we obtain
\[
G-x=\bigsqcup_{i=1}^q P_{n_i},
\qquad
G-N[x]=\bigsqcup_{i=1}^q P_{n_i-1}.
\]
Hence, by \Cref{lem:delete_contract} and multiplicativity on disjoint unions,
\[P_G(x)=\prod_{i=1}^q P_{P_{n_i}}(x)+x\prod_{i=1}^q P_{P_{n_i-1}}(x)\]
with the convention $P_{P_0}(x)=1$. 
\end{proof}

\begin{remark}\label{rem:big_star}
It follows from \Cref{lem:big_star} that
\[
P_G(-1)=\prod_{i=1}^q p_{n_i}-\prod_{i=1}^q p_{n_i-1},
\]
where \(p_m:=P_{P_m}(-1)\) as in \Cref{lem:path-minus-one}. 
From this we need only the following three consequences of \Cref{lem:path-minus-one}  to determine when $P_G(-1)$ is either $0, 1$ or $-1$:
\[
p_m=0 \iff m\equiv 1\pmod 3,
\]
\[
p_{m-1}=0 \iff m\equiv 2\pmod 3,
\]
and
\[
m\equiv 0\pmod 3 \implies p_m=p_{m-1}\in\{\pm1\}.
\]
\end{remark}

We start with the $P_G(-1)=0$ case.

\begin{thm}\label{thm:big_star_0}
Let $G:=G(n_1,\ldots,n_q)$ be a big star graph. Then \(P_G(-1)=0\)
if and only if either
\begin{enumerate}
    \item $n_i\equiv 0 \pmod 3$ for all $i$, or
    \item there exist $i\neq j$ such that $n_i\equiv 1 \pmod 3$ and $n_j\equiv 2 \pmod 3$.
\end{enumerate}
\end{thm}

\begin{proof}
Following the discussion from \Cref{rem:big_star}, we have $P_G(-1)=0$ if and only if
\[
\prod_{i=1}^q p_{n_i}
=
\prod_{i=1}^q p_{n_i-1}.
\]
Suppose first that there exist $i\neq j$ such that
$n_i\equiv 1\pmod 3$ and $n_j\equiv 2\pmod 3$.
Then $p_{n_i}=0$, so the first product is zero, and
$p_{n_j-1}=0$, so the second product is zero. Hence $P_G(-1)=0$. Next, suppose that $n_i\equiv 0\pmod 3$ for all $i$.
Then $p_{n_i}=p_{n_i-1}$ for every $i$. Therefore \(\prod_{i=1}^q p_{n_i}
=
\prod_{i=1}^q p_{n_i-1}\).
Hence again $P_G(-1)=0$.

Conversely, assume that $P_G(-1)=0$. If both products are zero, then some factor $p_{n_i}$ must be zero, so
$n_i\equiv 1\pmod 3$ for some $i$, and some factor $p_{n_j-1}$ must be zero, so
$n_j\equiv 2\pmod 3$ for some $j$. Thus condition (ii) holds. If both products are nonzero, then no $p_{n_i}$ is zero and no $p_{n_i-1}$ is zero.
Hence no $n_i\equiv 1\pmod 3$ and no $n_i\equiv 2\pmod 3$.
Therefore $n_i\equiv 0\pmod 3$ for all $i$, so condition (i) holds.

Thus $P_G(-1)=0$ if and only if either (i) or (ii) holds. As observed at the beginning,
this is equivalent to $\mathfrak a(G)<0$.
\end{proof}

It is known by \Cref{levit:tree_values} that  \(P_G(-1)\in\{-1,0,1\}\). After identifying when $P_G(-1)=0$, our goal is to determine for which big stars we have $P_G(-1)=1$ and $P_G(-1)=-1$. As we have seen in \Cref{thm:big_star_0},  mod $3$ data determines exactly when $P_G(-1)=0$ happens.  However, mod $3$ does not determine its sign when $P_G(-1)\neq 0$. As we will see, the sign is controlled by congruence classes modulo  $6$  together with a parity condition.

\begin{thm}\label{thm:big_star_pm1}
Let \(G=G(n_1,\ldots,n_q)\) be a big star graph.  
Let 
\[
c_k:=|\{ i : n_i\equiv k \pmod 6 \}|
\] for \(k\in\{0,1,2,3,4,5\}\). Then
\[
P_G(-1)=
\begin{cases}
(-1)^{c_2+c_3}, & \text{if } c_1=c_4=0 \text{ and } c_2+c_5>0,\\[2mm]
(-1)^{c_3+c_4+1}, & \text{if } c_2=c_5=0 \text{ and } c_1+c_4>0,\\[2mm]
0, & \text{otherwise}.
\end{cases}
\]
\end{thm}

\begin{proof}
As in  \Cref{rem:big_star}, we have \(P_G(-1)=\prod_{i=1}^q p_{n_i}-\prod_{i=1}^q p_{n_i-1}\)
where 
\[
\begin{array}{c|cccccc}
n_i \bmod 6 & 0 & 1 & 2 & 3 & 4 & 5\\ \hline
p_{n_i} & 1 & 0 & -1 & -1 & 0 & 1\\
p_{n_i-1} & 1 & 1 & 0 & -1 & -1 & 0
\end{array}
\]

Now consider the two products.

If \(c_1+c_4>0\), then \(\prod_i p_{n_i}=0\).  
If \(c_2+c_5>0\), then \(\prod_i p_{n_i-1}=0\).

So there are four cases to consider.

\textbf{Case 1:} \(c_1+c_4>0\) and \(c_2+c_5>0\).  
Then both products are zero and \(P_G(-1)=0\).

\textbf{Case 2:} \(c_1+c_4=0\) and \(c_2+c_5=0\).  
Then every \(n_i\equiv 0\) or \(3\pmod 6\) and $p_{n_i}=p_{n_i-1}\in\{\pm1\}$
for all \(i\). So the two products are equal. Therefore \(P_G(-1)=0\).

\textbf{Case 3:} \(c_1+c_4=0\) and \(c_2+c_5>0\).  
Then \(\prod_i p_{n_i-1}=0\), while
\(
\prod_{i=1}^q p_{n_i}=(-1)^{c_2+c_3}\)
because \(p_{n_i}=-1\) only when \(n_i\equiv 2\) or \(3\pmod 6\). Hence
\[
P_G(-1)=(-1)^{c_2+c_3}.
\]
\textbf{Case 4:} \(c_2+c_5=0\) and \(c_1+c_4>0\).  
Then \(\prod_i p_{n_i}=0\), while
\(\prod_{i=1}^q p_{n_i-1}=(-1)^{c_3+c_4}\)
because \(p_{n_i-1}=-1\) only  when \(n_i\equiv 3\) or \(4\pmod 6\). Hence
\[
P_G(-1)=-(-1)^{c_3+c_4}=(-1)^{c_3+c_4+1}. \qedhere
\]
\end{proof}

The following characterizes when  $P_G(-1)=1$ and when $P_G(-1)=-1$ .

\begin{cor}\label{cor:big_star_pm1_sign}
Let \(G=G(n_1,\ldots,n_q)\) be a big star graph.

    \begin{enumerate}    
    \item \(P_G(-1)=1\) if and only if either
    \begin{enumerate}
        \item \(n_i\not\equiv 1\pmod 3\) for all \(i\), at least one \(n_i\equiv 2\pmod 3\) and \(
        c_2+c_3 \) is even, or
        \item \(n_i\not\equiv 2\pmod 3\) for all \(i\), at least one \(n_i\equiv 1\pmod 3\) and \(c_3+c_4\)
 is odd.
    \end{enumerate}
 
     \item \(P_G(-1)=-1\) if and only if either
    \begin{enumerate}
        \item \(n_i\not\equiv 1\pmod 3\) for all \(i\), at least one \(n_i\equiv 2\pmod 3\) and \( c_2+c_3 \)
is odd, or
        \item \(n_i\not\equiv 2\pmod 3\) for all \(i\), at least one \(n_i\equiv 1\pmod 3\) and \( c_3+c_4 \) is even.
    \end{enumerate}
        \end{enumerate}
\end{cor}

We now determine the independence number of big star graphs.

\begin{lem}\label{lem:alpha_big_star}
Let $G(n_1,\ldots,n_q)$ be a big star graph. Then
\[
\alpha\bigl(G(n_1,\ldots,n_q)\bigr)
=
\sum_{i=1}^q \left\lfloor n_i/2\right\rfloor+\max\{1,r\}
\]
where  \(r=|\{ i : n_i \text{ is odd} \}|\).
\end{lem}

\begin{proof}
Write $G:=G(n_1,\ldots,n_q)$ with the center vertex $x$.  A maximum independent set of $G$ either contains $x$, or does not contain $x$. If it contains $x$, then on the $i$th arm we may choose at most
\[
\alpha(P_{n_i-1})=\left\lceil \frac{n_i-1}{2}\right\rceil=\left\lfloor \frac{n_i}{2}\right\rfloor
\]
additional vertices. Hence the largest such independent set has size \(1+\sum_{i=1}^q \left\lfloor n_i/2\right\rfloor\). If it does not contain $x$, then on the $i$th arm we may choose at most \(\alpha(P_{n_i})=\left\lceil n_i/2\right\rceil\)
vertices. Hence the largest such independent set has size \(\sum_{i=1}^q \left\lceil n_i/2\right\rceil\). Therefore
\[
\alpha(G)
=
\max\left\{
1+\sum_{i=1}^q \left\lfloor n_i/2\right\rfloor, 
\sum_{i=1}^q \left\lceil n_i/2\right\rceil
\right\}=
\sum_{i=1}^q \left\lfloor n_i/2\right\rfloor+\max\{1,r\}
\]
where the last equality follows from  \(\sum_{i=1}^q \left\lceil n_i/2\right\rceil
=
\sum_{i=1}^q \left\lfloor n_i/2\right\rfloor+r\).
\end{proof}

We are now ready to describe which big star graphs are pseudo-Gorenstein\(^*\) using \Cref{thm:big_star_pm1,lem:alpha_big_star}.

\begin{cor}\label{cor:pseudogorenstein_star_big_star}
Let \(G=G(n_1,\ldots,n_q)\) be a big star graph.
Then \(G\) is pseudo-Gorenstein\(^*\) if and only if one of the following holds:

\begin{enumerate}
    \item \(n_i\not\equiv 1\pmod 3\) for all \(i\), at least one \(n_i\equiv 2\pmod 3\), and
    \(
    \alpha(G)
    \equiv
    c_2+c_3
    \pmod 2\);
    \item \(n_i\not\equiv 2\pmod 3\) for all \(i\), at least one \(n_i\equiv 1\pmod 3\), and
    \(
\alpha(G)
    \equiv
    c_3+c_4+1
    \pmod 2\).
\end{enumerate}
\end{cor}

Lastly, we determine which big stars have symmetric independence polynomials. Surprisingly, exactly one big star has this property.

\begin{thm}\label{prop:big_star_symmetry}
Let \(G=G(n_1,\ldots,n_q)\) be a big star. Then \(P_G(x)\) is symmetric if and only if $G\cong G(1,1,5)$.
\end{thm}

\begin{proof}
Let $\alpha:=\alpha(G)$. If \(P_G(x)\) is symmetric, then \(g_{\alpha}=g_0=1\). This means \(G\) has a unique
maximum independent set.

Recall that  \(r=|\{ i : n_i \text{ is odd} \}|\).  If \(r=1\), then there is a maximum independent set containing the center and also one
avoiding the center, contradiction. If \(2\le r\le q-1\), then every maximum independent
set avoids the center, and each even arm \(P_{2a}\) contributes at least two maximum
choices, again a contradiction. Hence either all arms are even or all arms are odd.

Suppose first that \(n_i=2a_i\) for all \(i\). In this case, $g_1=1+2\sum_{i=1}^q a_i$ and    $\alpha=1+\sum_{i=1}^q a_i$ by \Cref{lem:alpha_big_star}. An independent set of size \(\alpha-1\) either avoids the center or contains it. If it avoids the center, then on the \(i\)th arm $x^{(i)}_1-x^{(i)}_2-\cdots-x^{(i)}_{2a_i}$
we must choose a maximum independent set of \(P_{2a_i}\). Thus each arm $P_{2a_i}$ contribute exactly $a_i$ vertices and the number of independent sets of size $a_i$ in $P_{2a_i}$ is $\binom{2a_i+1-a_i}{a_i}=a_i+1$. Thus these contribute $\prod_{i=1}^q (a_i+1)$.

If it contains the center, then \(x^{(i)}_1\) is forbidden on every arm. So on the \(i\)th arm we may only choose vertices from $x^{(i)}_2-x^{(i)}_3-\cdots-x^{(i)}_{2a_i}\cong P_{2a_i-1}$.
Since the center already contributes one vertex, the arms together must contribute \(\sum_{i=1}^q a_i-1\). Hence exactly one arm, say the \(j\)th, contributes \(a_j-1\) vertices, while every other arm contributes \(a_i\), necessarily the unique maximum independent set $\{x^{(i)}_2,x^{(i)}_4,\ldots,x^{(i)}_{2a_i}\}$. 
On the exceptional \(j\)th arm, the number of independent sets of size \(a_j-1\) in \(P_{2a_j-1}\) is $\binom{(2a_j-1) +1- (a_j-1)}{a_j-1}= \binom{a_j+1}{2}$. 
Thus these contribute $\sum_{j=1}^q \binom{a_j+1}{2}$. Therefore
\[
g_{\alpha-1}
=
\prod_{i=1}^q (a_i+1)+\sum_{i=1}^q \binom{a_i+1}{2}.
\]
Since \(\binom{a_i+1}{2}\ge a_i\) and, because \(q\ge 3\) and each \(a_i\ge 1\), we have $\prod_{i=1}^q (a_i+1)>1+\sum_{i=1}^q a_i$.  Thus, $g_{\alpha-1}>1+2\sum_{i=1}^q a_i=g_1$,
contradicting symmetry. This means all arms are odd. 

Let \(n_i=2a_i+1\) for all \(i\). By \Cref{lem:alpha_big_star}, we have $\alpha
=q+\sum_{i=1}^q a_i$. Now let \(I\) be an independent set containing the center \(x\). Then \(x^{(i)}_1\notin I\) for every \(i\), so on the \(i\)th arm we may choose vertices only from $x^{(i)}_2-x^{(i)}_3-\cdots-x^{(i)}_{2a_i+1}\cong P_{2a_i}$. Hence at most \(a_i\) vertices from that arm. Therefore, since \(q\ge 3\), we get $|I|\le   1+\sum_{i=1}^q a_i
\le \alpha-2$. Thus no independent set of size \(\alpha-1\) contains the center. It follows that every independent set counted by \(g_{\alpha-1}\) avoids the center. On the \(i\)th arm, the maximum possible contribution is \(a_i+1\).  Hence such a set must be maximal on all but exactly one arm, and on the remaining arm, say the \(i\)th, it must have size \(a_i\). For \(j\neq i\), the unique maximum independent set of \(P_{2a_j+1}\) is
\( \{x^{(j)}_1,x^{(j)}_3,\ldots,x^{(j)}_{2a_j+1}\}
\),
while the number of independent sets of size \(a_i\) in \(P_{2a_i+1}\) is
\(\binom{(2a_i+1)+1-a_i}{a_i}=\binom{a_i+2}{2}\).
Summing over the choice of the exceptional arm gives
\[
g_{\alpha-1}=\sum_{i=1}^q \binom{a_i+2}{2}.
\]
Since \(P_G(x)\) is symmetric of degree \(\alpha\), we have  
\[
g_{\alpha-1}= \sum_{i=1}^q \binom{a_i+2}{2}
=
q+1+2\sum_{i=1}^q a_i=g_1.
\]
Since $\binom{a_i+2}{2}=\binom{a_i}{2}+2a_i+1$,
the equality becomes $\sum_{i=1}^q \binom{a_i}{2}=1$. This equality forces exactly one term to be \(1\) and all the others to be \(0\). Hence, after reordering, exactly one \(a_i\) is \(2\), and every other \(a_j\) is \(0\) or \(1\). Since \(n_i=2a_i+1\), this means that exactly one arm has length \(5\), and every other arm has length \(1\) or \(3\).

If \(q=3\), then the only possibilities are \(G(1,1,5)\), \(G(1,3,5)\), and \(G(3,3,5)\),
and a direct computation shows that only \(G(1,1,5)\) is symmetric.

Now assume \(q\ge 4\). From the previous paragraph, after reordering we may write
\[
G\cong G(\underbrace{1,\ldots,1}_{b},\underbrace{3,\ldots,3}_{t},5)
\]
where $q=b+t+1$ and $\alpha=b+2t+3$.

Set $S(x):=(1+x)^b(1+3x+x^2)^t$. 
Then, we have 
\[
P_G(x)=S(x)\underbrace{(1+5x+6x^2+x^3)}_{A(x)}+\underbrace{x(1+4x+3x^2)(1+2x)^t}_{B(x)}.
\]
Since \(S(x)\) is symmetric of degree \(b+2t=\alpha-3\), we have $x^{\alpha-3}S(x^{-1})=S(x)$. Then
\[
[x^{\alpha-k}]S(x)A(x)=[x^k]S(x)x^3A(x^{-1}).
\]
Therefore
\[
[x^k]S(x)A(x)-[x^{\alpha-k}]S(x)A(x)
=
[x^k]S(x)\bigl(A(x)-x^3A(x^{-1})\bigr)
\]
where  $A(x)-x^3A(x^{-1})= x(x-1)$.
Hence
\[
[x^k]S(x)A(x)-[x^{\alpha-k}]S(x)A(x)
=
[x^k]\bigl(x(x-1)S(x)\bigr).
\]

Next, $\deg B=t+3=\alpha-(q-1)$.
So if \(k\le q-2\), then \(\alpha-k>\deg B\). Thus $[x^{\alpha-k}]B(x)=0$.
It follows that for every \(k\le q-2\), we have
\[
g_k-g_{\alpha-k}
=
[x^k]\Bigl(B(x)+x(x-1)S(x)\Bigr).
\]
Consider \(k=2\). One can deduce the coefficients of $x^2$ in $B(x)$ and $x(x-1)S(x)$ to conclude that $g_2-g_{\alpha-2}=6-q$.
If \(P_G(x)\) is symmetric, then $q=6$ and $b+t=5$.

Now consider \(k=3\) and make use of \(b=5-t\) in deducing the coefficient of $x^3$ in  $B(x)$ and $x(x-1)S(x)$. So, we have $g_3-g_{\alpha-3}=t-2$.
Symmetry forces \(t=2\). Hence \(b=3\). Thus the only  possibility with \(q\ge 4\) is $G\cong G(1,1,1,3,3,5)$.
A direct computation gives the following polynomial, which is not symmetric:
\[
P_G(x)
=
1+15x+91x^2+296x^3+577x^4+714x^5+575x^6+296x^7+91x^8+15x^9+x^{10}.
\]
Therefore \(P_G(x)\) is symmetric if and only if \(G\cong G(1,1,5)\).
\end{proof}

\section{Whiskering and Independence Polynomials}\label{sec:whisker}

In this section, we study the effect of whiskering, that is, attaching leaves to select vertices, on the independence polynomial. We derive a general formula for the independence polynomial, obtain a simple expression for its value at \(-1\), and show that if every vertex of the base graph receives at least one leaf, then symmetry occurs exactly when each vertex receives two leaves.

\begin{prop}\label{prop:leaf-attachment}
Let \(H\) be a finite simple graph on vertices \(x_1,\dots,x_n\), and let
\(f_1,\dots,f_n\) be nonnegative integers. Let \(G\) be the graph obtained from \(H\)
by attaching \(f_i\) leaf vertices to \(x_i\) for each \(i\in[n]\). Set \(C:=\{ x_i : f_i>0 \}\subseteq V(H)\).
Then
\[
P_G(x)=
\sum_{\substack{S\subseteq V(H)\\ S \text{ is independent in }H}}
x^{|S|}\prod_{x_i\notin S}(1+x)^{f_i}.
\]
Moreover, we have the following statements:
\begin{enumerate}
\item[\textup{(a)}] If \(C\) is not an independent set of \(H\), then \(P_G(-1)=0\).
\item[\textup{(b)}] If \(C\) is an independent set of \(H\), then \(P_G(-1)=(-1)^{|C|}P_{H-N_H[C]}(-1)\).
\item[\textup{(c)}] One has
\(
\alpha(G)=\sum_{i=1}^n f_i+\alpha(H-C) \).
\end{enumerate}
\end{prop}

\begin{proof}
Let \(L_i\) denote the set of leaves attached to \(x_i\) and \(|L_i|=f_i\).
Given an independent set \(I\) of \(G\), let $S:=I\cap V(H)=\{ x_i\in V(H): x_i\in I \}$.
Then $S$ is an independent set of \(H\). Conversely, if
$S$ is independent in \(H\), then for each \(i\in S\) no vertex of
\(L_i\) may be chosen, while for each \(i\notin S\) any subset of \(L_i\) may be
chosen. Hence the total contribution of all independent sets \(I\) with this fixed
set \(S\) is
\[
x^{|S|}\prod_{x_i\notin S}(1+x)^{f_i},
\]
which proves the displayed formula for \(P_G(x)\).

Now focus on $P_G(-1)$. Notice that for an independent set $S$ of $H$, whenever  $i\in C$ (i.e. $f_i>0$) but $i\notin S$, then the corresponding summand for $S$   is zero. So, for a summand to survive, every $x_i\in C$ must be contained in $S$.  So, a summand is nonzero if and only if \(f_i=0\) for every
 \(i\notin S\), i.e. if and only if \(C\subseteq S\). Thus, if \(C\) is not
independent in \(H\), then no independent set \(S\) can contain \(C\). Hence
\(P_G(-1)=0\). This proves \textup{(a)}.

Assume now that \(C\) is independent in \(H\). Then the surviving sets \(S\) are
exactly those of the form $S=C\sqcup T$
where \(T\) is an independent set of \(H-N_H[C]\). Therefore
\[
P_G(-1)
=
\sum_{T\in \operatorname{Ind}(H-N_H[C])}(-1)^{|C|+|T|}
=
(-1)^{|C|}P_{H-N_H[C]}(-1),
\]
proving \textup{(b)}.

Finally, let \(I\) be a maximum independent set of \(G\). If \(x_i\in I\) for some \(i\) with
\(f_i>0\), then replacing \(x_i\) by all vertices of \(L_i\) preserves independence and
does not decrease cardinality. Repeating this for each \(i\), we obtain a maximum
independent set \(J\) containing no vertex of \(C\). Since each leaf in \(L_i\) is
adjacent only to \(x_i\), maximality forces \(L_i\subseteq J\) for every \(i\) with
\(f_i>0\). Thus \(J\) consists of all \(\sum_i f_i\) attached leaves together with an
independent set of \(H-C\). Conversely, any independent set of \(H-C\), together with
all attached leaves, is independent in \(G\). Hence
\[
\alpha(G)=\sum_{i=1}^n f_i+\alpha(H-C),
\]
which proves \textup{(c)}.
\end{proof}

In what follows, we discuss when the whiskering operation yields a symmetric independence polynomial. This result generalizes and recovers one of the constructions of Stevanović's  from \cite{Stevanovic} guaranteeing symmetric independence polynomials.

\begin{thm}\label{thm:full_leaf_attachment_symmetric}
Assume in the setting of \Cref{prop:leaf-attachment} that \(f_i\geq 1\) for every
\(i\in [n]\), so \(C=V(H)\). Then \(P_G(x)\) is symmetric if and only if
\[
f_1=\cdots=f_n=2.
\]
\end{thm}

\begin{proof}
Set $F:=\sum_{i=1}^n f_i$. Since \(C=V(H)\), \Cref{prop:leaf-attachment}\textup{(c)} yields $\alpha(G)=F$. Again by \Cref{prop:leaf-attachment},
\begin{equation}\label{eq:full-leaf-attachment}
P_G(x)=\sum_{S\in \Ind(H)} C_S(x)
\end{equation}
where $C_S(x):=x^{|S|}(1+x)^{F-\sum_{x_i\in S}f_i}$ and  $\Ind(H)$ denotes the set of all independent sets of $H$.
Then
\[
\deg C_S=F-\sum_{x_i\in S}(f_i-1).
\]
Let $R:=\{x_i\in V(H): f_i=1\}$.
Then \(\deg C_S=F\) if and only if \(S\subseteq R\). Let $H|_R$ denote the induced subgraph of $H$ on $R$. Hence the coefficient of
\(x^F\) in \(P_G(x)\) (the number of independent sets of size $F$) is
\[
g_F=|\Ind(H|_R)|=P_{H|_R}(1).
\]
 If \(P_G(x)\) is symmetric, then \(g_F=g_0=1\). Since every graph on at least
one vertex has at least two independent sets, namely \(\varnothing\) and any
singleton, this forces \(R=\varnothing\). Thus \(f_i\geq 2\) for all \(i\).

Now let $T:=\{x_i\in V(H): f_i=2\}$.
Since \(f_i\geq 2\) for all \(i\), the only summand in
\eqref{eq:full-leaf-attachment} of degree \(F\) is the one corresponding to
\(S=\varnothing\), namely \((1+x)^F\). Moreover, \(\deg C_S=F-1\) if and only
if \(S=\{x_i\}\) for some \(x_i\in T\). Therefore the coefficient of \(x^{F-1}\) in \(P_G(x)\) receives contributions only
from the summand \(C_\varnothing(x)=(1+x)^F\) and from the summands
\(C_{\{x_i\}}(x)\) with \(x_i\in T\). The contribution from \(C_\varnothing(x)\) is $\binom{F}{F-1}=F$. For each \(x_i\in T\), we have $C_{\{x_i\}}(x)=x(1+x)^{F-2}$. Since there are \(|T|\) such vertices, these terms contribute a total of \(|T|\).
Thus
\[
g_{F-1}=F+|T|.
\]
On the other hand, \(g_1=|V(G)|=n+F\). If \(P_G(x)\) is symmetric, then
\(g_{F-1}=g_1\). So
\[
F+|T|=n+F.
\]
Hence \(|T|=n\), which means \(f_i=2\) for every \(i\).

Conversely, if \(f_i=2\) for every \(i\), then \(F=2n\), and
\eqref{eq:full-leaf-attachment} becomes
\[
P_G(x)
=
\sum_{S\in \Ind(H)} x^{|S|}(1+x)^{2(n-|S|)}
=
(1+x)^{2n}
P_H\!\left(\frac{x}{(1+x)^2}\right).
\]
Thus
$$
x^{2n}P_G(1/x)
=
x^{2n}\left(1+\frac1x\right)^{2n}
P_H\!\left(\frac{1/x}{(1+1/x)^2}\right) =
(1+x)^{2n}P_H\!\left(\frac{x}{(1+x)^2}\right)
=
P_G(x).
$$
Therefore \(P_G(x)\) is symmetric.
\end{proof}

\subsection{An application: caterpillar graphs}

As an application of whiskering, we now focus on caterpillar graphs. Since a caterpillar is obtained from a path by attaching leaves along the central spine, the general results of this section translate into explicit formulas in terms of the positions of the legs. Encoding these positions by the gaps between consecutive support vertices allows us to express \(P_G(-1)\) as a product of path contributions, from which the criteria for the values \(0\), \(1\), and \(-1\) follow immediately. We also provide a characterization of pseudo-Gorenstein\(^*\) caterpillars.

\begin{defn}
A graph \(G\) is called a \emph{caterpillar} if \(G\) is a tree and there exists
a path \(x_1,\dots,x_n\), called a \emph{central path}, such that every vertex of
\(G\) is either one of the \(x_i\)'s or is adjacent to one of them. Every vertex
of \(G\) that does not lie on the central path is called a \emph{leg}.
\end{defn}

\begin{notation}\label{not:caterpillar}
Let \(G\) be a caterpillar with central path $x_1x_2\cdots x_n$.
For each \(i\in [n]\), let \(L_i\) denote the set of all legs adjacent to \(x_i\), i.e.
\[
L_i:=\{ v\in V(G)\setminus \{x_1,\dots,x_n\} : \{v,x_i\}\in E(G) \}.
\]
Set \(f_i:=|L_i|\). Thus
\(G\) is obtained from \(P_n\) by attaching \(f_i\) leaves to \(x_i\) for each
\(i\in [n]\).

Let $B:=\{ x_i: f_i>0 \}$.
Thus \(B\) records the positions on the central path that carry at least one leg.

If \(B=\varnothing\), then \(G=P_n\). In this case, we set $r:=0, ~m_0:=n$ and $\ell_0:=n$. Now assume \(B\neq \varnothing\) and write $B=\{x_{b_1}<\cdots<x_{b_r} ~:~ b_1<\cdots<b_r\}$.

The vertices of the central path carrying no legs form \(r+1\) consecutive
gaps. Let $m_0,m_1,\dots,m_r$
denote their lengths where:
\begin{itemize}
\item \(m_0\) is the number of vertices before \(x_{b_1}\);
\item \(m_r\) is the number of vertices after \(x_{b_r}\);
\item for \(1\le j\le r-1\), \(m_j\) is the number of vertices strictly between
\(x_{b_j}\) and \(x_{b_{j+1}}\).
\end{itemize}
Equivalently,
\[
m_0=b_1-1,\qquad
m_j=b_{j+1}-b_j-1 \ \ (1\le j\le r-1),\qquad
m_r=n-b_r.
\]

Next, delete from the central path every vertex \(x_{b_j}\) together with its
neighbors on the central path. The remaining vertices again form \(r+1\) path
segments. Let $\ell_0,\ell_1,\dots,\ell_r$
denote their lengths. Equivalently,
\[
\ell_0=\max\{m_0-1,0\},\qquad
\ell_j=\max\{m_j-2,0\}\ \ (1\le j\le r-1),\qquad
\ell_r=\max\{m_r-1,0\}.
\]
\end{notation}

\begin{thm}\label{thm:caterpillar-minus-one}
Let \(G\) be a caterpillar as in \Cref{not:caterpillar}. Then
\[
P_G(-1)=
\begin{cases}
0, & \text{if } b_{j+1}=b_j+1 \text{ for some }j,\\[4pt]
(-1)^r\displaystyle\prod_{j=0}^r p_{\ell_j}, & \text{otherwise},
\end{cases}
\]
where \(p_m:=P_{P_m}(-1)\) is as in \Cref{lem:path-minus-one}.
\end{thm}

\begin{proof}
Apply \Cref{prop:leaf-attachment} to the central path \(H=P_n\) and recall that $B=\{x_{b_1},\dots,x_{b_r}\}$.
If two elements of \(B\) are consecutive, then \(B\) is not independent in \(P_n\). Thus
\(P_G(-1)=0\) by \Cref{prop:leaf-attachment}\textup{(a)}. Assume now that no two elements of \(B\) are consecutive. Then \(B\) is an independent
set of \(P_n\). By construction, deleting \(N_{P_n}[B]\) leaves exactly the \(r+1\)
path segments of lengths \(\ell_0,\dots,\ell_r\). Hence
\[
P_n-N_{P_n}[B]\cong \bigsqcup_{j=0}^r P_{\ell_j}.
\]
Therefore \Cref{prop:leaf-attachment}\textup{(b)} and multiplicativity on disjoint
unions give
\[
P_G(-1)=(-1)^r\prod_{j=0}^r P_{P_{\ell_j}}(-1)
=
(-1)^r\prod_{j=0}^r p_{\ell_j}.
\]
\end{proof}

\begin{cor}\label{cor:caterpillar-zero}
Let \(G\) be a caterpillar as in \Cref{not:caterpillar}. Then $P_G(-1)=0$ if and only if  either  $b_{j+1}=b_j+1$ for some $j$, or $\ell_j\equiv 1\pmod 3$ for some $j$.

Equivalently, $\mathfrak a(G)=0$ if and only if no two vertices of $B$ are consecutive and $\ell_j\not\equiv 1\pmod 3$ for all $j$.
\end{cor}

\begin{proof}
By \Cref{thm:caterpillar-minus-one}, the nontrivial case is when no two elements of
\(B\) are consecutive, in which case
\[
P_G(-1)=(-1)^r\prod_{j=0}^r p_{\ell_j}.
\]
Now \Cref{lem:path-minus-one} says that \(p_m=0\) if and only if \(m\equiv 1\pmod 3\).
This proves the first assertion. The second statement follows from the fact  $\mathfrak a(G)=0$ whenever $P_G(-1)\neq 0$.
\end{proof}

\begin{cor}\label{cor:caterpillar-sign}
Let \(G\) be a caterpillar as in \Cref{not:caterpillar}, and define
\[
\lambda:=\Big|\{ j\in\{0,\dots,r\}: \ell_j\equiv 2,3 \pmod 6 \}\Big|.
\]
If \(P_G(-1)\neq 0\), then $P_G(-1)=(-1)^{r+\lambda}$.
\end{cor}

\begin{proof}
By \Cref{cor:caterpillar-zero}, the assumption \(P_G(-1)\neq 0\) means that each
\(p_{\ell_j}\in\{\pm 1\}\). By \Cref{lem:path-minus-one}, one has $p_{\ell_j}=-1 $ if and only if $\ell_j\equiv 2,3\pmod 6$. Thus $\prod_{j=0}^r p_{\ell_j}=(-1)^{\lambda}$  and \Cref{thm:caterpillar-minus-one} yields
\[
P_G(-1)=(-1)^r(-1)^\lambda=(-1)^{r+\lambda}.
\]
\end{proof}

\begin{cor}\label{cor:alpha-caterpillar}
Let \(G\) be a caterpillar as in \Cref{not:caterpillar}. Then
\[
\alpha(G)=\sum_{i=1}^n f_i+\sum_{j=0}^r \left\lceil \frac{m_j}{2}\right\rceil.
\]
\end{cor}

\begin{proof}
By \Cref{prop:leaf-attachment}\textup{(c)} with \(H=P_n\), we have $\alpha(G)=\sum_{i=1}^n f_i+\alpha(P_n-B)$. Since \(P_n-B\cong \bigsqcup_{j=0}^r P_{m_j}\), we have
\[
\alpha(P_n-B)=\sum_{j=0}^r \alpha(P_{m_j})
=\sum_{j=0}^r \left\lceil \frac{m_j}{2}\right\rceil. \qedhere
\]
\end{proof}

The following immediately follows from \Cref{cor:caterpillar-zero,cor:caterpillar-sign}.
\begin{cor}\label{cor:caterpillar_pm1_sign}
Let \(G\) be a caterpillar as in \Cref{not:caterpillar}. Then:
\begin{enumerate}
    \item \(P_G(-1)=1\) if and only if
    \begin{enumerate}
        \item no two consecutive vertices of the central path both carry legs;
        \item \(\ell_j\not\equiv 1 \pmod 3\) for every \(0\le j\le r\); and
        \item \(r+\lambda\) is even.
    \end{enumerate}

    \item \(P_G(-1)=-1\) if and only if
    \begin{enumerate}
        \item no two consecutive vertices of the central path both carry legs;
        \item \(\ell_j\not\equiv 1 \pmod 3\) for every \(0\le j\le r\); and
        \item \(r+\lambda\) is odd.
    \end{enumerate}
\end{enumerate}
\end{cor}

Lastly, we can conclude when a caterpillar is pseudo-Gorenstein\(^{*}\) by applying \Cref{cor:caterpillar-sign,cor:alpha-caterpillar}.

\begin{cor}\label{cor:pseudogorenstein_star_caterpillar}
Let \(G\) be a caterpillar as in \Cref{not:caterpillar}. Then \(G\) is
pseudo-Gorenstein\(^{*}\) if and only if all of the following hold:
\begin{enumerate}
    \item no two vertices of \(B\) are consecutive;
    \item  \(\ell_j\not\equiv 1\pmod 3\) for all \(j=0,\dots,r\);
    \item \(\sum_{x_i\in B} f_i+\sum_{j=0}^{r}\left\lceil \frac{m_j}{2}\right\rceil
    \equiv
    r+\lambda
    \pmod 2\).
\end{enumerate}
\end{cor}

We conclude this section by discussing when caterpillar graphs have symmetric independence polynomials. The following result follows immediately from  \Cref{thm:full_leaf_attachment_symmetric}.

\begin{cor}
 Let \(G\) be a caterpillar as in \Cref{not:caterpillar} such that $B=\{x_1,\ldots, x_n\}$. Then $P_G(x)$ is symmetric if and only if $f_i=2$ for each $i$.    
\end{cor}

\section{Independence polynomial of cochordal graphs}

In this section, we study independence polynomials of cochordal graphs. Using the quasi-forest structure of clique complexes of chordal graphs and the associated \(b\)-sequence from \cite{Combinatorica_cochordal}, we classify the symmetric case, show that symmetry implies unimodality, and record formulas for \(P_G(-1)\) and for the multiplicity of the root \(-1\) in terms of graph invariants.

We first recall the definition of a quasi-forest.

\begin{defn}
Let \(\Delta\) be a simplicial complex with facets \(F_1,\dots,F_s\).

A facet \(F\) of \(\Delta\) is called a \emph{leaf} if either \(F\) is the unique facet of \(\Delta\), or there exists a facet \(B\neq F\) of \(\Delta\) such that
\[
F\cap H \subseteq F\cap B
\]
for every facet \(H\neq F\) of \(\Delta\). In this case, \(B\) is called a \emph{branch} of \(F\).

The simplicial complex \(\Delta\) is called a \emph{quasi-forest} if its facets can be ordered
\[
F_1,\dots,F_s
\]
in such a way that, for each \(i=1,\dots,s\), the facet \(F_i\) is a leaf of the subcomplex
generated by \(F_1,\dots,F_i\).

If, in addition, \(\Delta\) is connected, then \(\Delta\) is called a \emph{quasi-tree}.
An ordering \(F_1,\dots,F_s\) with the above property is called a \emph{leaf order}.
\end{defn}

Next we recall the relationship between chordal graphs and quasi-forests.

\begin{thm}\label{chordal_quasi-forest}\cite[Theorem 9.2.12]{monomial_ideals_book}
  Given a finite simple graph $G$ on $n$ vertices that is not complete, the following are equivalent:
  \begin{enumerate}
      \item[(i)] $G$ is chordal;
      \item[(ii)] the clique complex of $G$ is a quasi-forest.
  \end{enumerate}
\end{thm}

We investigate when the independence polynomials of cochordal graphs are symmetric.

\begin{remark}
    Notice that when $\alpha(G)=1$, the independence polynomial of $G$ is $P_G(x)=1+|V(G)|x$. So, for $P_G(x)$ to be symmetric, we need $|V(G)|=1$. Thus symmetry forces  $G=K_1$.
\end{remark}

\begin{thm}
Let $G$ be a cochordal graph and \(d=\alpha(G)\ge 2\).
Then $P_G(x)$ is symmetric if and only if there exists an integer $m\ge 0$
such that
\[
P_G(x)=(1+x)^d+mx(1+x)^{d-2}.
\]
Moreover, the independence polynomial is unimodal when it is symmetric. Lastly, for every $m\ge 0$, the graph
\[
(d-2)K_1 \sqcup (K_{m+2}-e)
\]
is cochordal and has this independence polynomial.
\end{thm}

\begin{proof}
Let $\Delta$ be the independence complex of $G$. Since $G$ is
cochordal, $\overline{G}$ is chordal, and $\Delta$ is the clique complex of
$\overline{G}$. Hence $\Delta$ is a quasi-forest by
\cite[Lemma~1.1]{HHZ-CM-chordal}. In particular, $\dim \Delta=d-1$, and for
$1\le i\le d$ we have \(g_i=f_{i-1}(\Delta)\).

By \cite[Theorem~1.1]{Combinatorica_cochordal}, there exist positive integers $b_1,\dots,b_d$ such that
\[
\sum_{i=1}^d g_i(x-1)^{i-1}=\sum_{i=1}^d b_i x^{i-1}.
\]
Then
\begin{equation}\label{eq:PG-b-sequence}
P_G(x)=1+\sum_{i=1}^d g_ix^i=1+x\sum_{i=1}^d b_i(1+x)^{i-1}.
\end{equation}

From \eqref{eq:PG-b-sequence}, the coefficients of $x^d$, $x^{d-1}$, and $x$
in $P_G(x)$ are
\[
g_d=b_d,\qquad
g_{d-1}=b_{d-1}+(d-1)b_d,\qquad
g_1=\sum_{i=1}^d b_i.
\]
Assume that $P_G(x)$ is symmetric. Since $g_0=1$, symmetry gives $g_d=g_0=b_d=1$. 

Symmetry also gives $g_{d-1}=g_1$. So
\[
b_{d-1}+(d-1)b_d=\sum_{i=1}^d b_i.
\]
Since $b_d=1$, this becomes \(\sum_{i=1}^{d-2} b_i=d-2\).
Since each $b_i$ is positive, it follows that \(b_1=\cdots=b_{d-2}=1\).

Write $b_{d-1}=m+1$ with $m\ge 0$. Substituting into
\eqref{eq:PG-b-sequence}, we get
\[
\begin{aligned}
P_G(x)
&=1+x\Bigl(\sum_{i=1}^{d-2}(1+x)^{i-1}+(m+1)(1+x)^{d-2}+(1+x)^{d-1}\Bigr)\\
&=1+\bigl((1+x)^{d-2}-1\bigr)+(m+1)x(1+x)^{d-2}+x(1+x)^{d-1}\\
&=(1+x)^{d-2}\bigl(1+(m+2)x+x^2\bigr)\\
&=(1+x)^d+mx(1+x)^{d-2}.
\end{aligned}
\]
Notice that  the polynomial $P_G(x)= (1+x)^{d-2}\bigl(1+(m+2)x+x^2\bigr)$ is unimodal. To see this, note that $1+(m+2)x+x^2$ has two real negative roots for $m\ge 0$. Hence $P_G(x)$ is real-rooted with nonnegative coefficients. Thus, it is log-concave and therefore unimodal.

Conversely, suppose that \(P_G(x)=(1+x)^d+mx(1+x)^{d-2}\)
for some $m\ge 0$. Then  $P_G(x)$ is symmetric as we see below:
\[
x^dP_G(1/x)
= x^d\left((1+1/x)^d+m(1/x)(1+1/x)^{d-2}\right)
= (1+x)^d+mx(1+x)^{d-2}
= P_G(x).
\]
Finally, consider the graph $G$ obtained from deleting an edge, say $e$, from $K_{m+2}$. Then
\[
P_G(x)=P_{K_{m+2}-e}(x)=1+(m+2)x+x^2.
\]
Now consider the graph $H$ that is the disjoint union of $G$ and $(d-2)$ isolated vertices $(d-2)K_1$. So $H=(d-2)K_1 \sqcup (K_{m+2}-e)$. 
Hence
\[
P_{H}(x)
=(1+x)^{d-2}\bigl(1+(m+2)x+x^2\bigr)
=(1+x)^d+mx(1+x)^{d-2}.
\]
Moreover, $\overline{G}=\overline{K_{m+2}-e}$ is an edge together with $m$ isolated
vertices. Hence chordal. Adding $(d-2)$ isolated vertices on the $G$
corresponds to adding $(d-2)$  vertices in the complement and they are connected to every vertex in $G$, which
preserves chordality. Therefore $H$ is cochordal and its independence polynomial coincides with the desired symmetric form.
\end{proof}

\begin{cor}\label{cor:cochordal-k-components-minus-one}
Let \(G\) be a finite simple graph, and suppose that \(\overline{G}\) is chordal with
\(k\) connected components. Then \(P_G(-1)=1-k\).
\end{cor}

\begin{proof}
Set \(H:=\overline{G}\). For each \(i\ge 1\), let \(s(K_i,H)\) denote the number of
cliques of size \(i\) in \(H\). Since independent sets of \(G\) are exactly cliques of
\(H\), we have
\[
P_G(x)=1+\sum_{i\ge 1}s(K_i,H)x^i.
\]
Thus, we obtain
\[
P_G(-1)=1+\sum_{i\ge 1}(-1)^i s(K_i,H).
\]
Since \(H\) is chordal, it follows from \cite[Lemma 2]{Klusowski_chordal_clique} that
\[
\operatorname{cc}(H)=\sum_{i\ge 1}(-1)^{i+1}s(K_i,H).
\]
Hence \(P_G(-1)=1-\operatorname{cc}(H)=1-k\).
\end{proof}

\begin{remark}
\Cref{cor:cochordal-k-components-minus-one} is a direct reformulation of
\cite[Lemma 2]{Klusowski_chordal_clique}. As noted there, the clique count identity for
chordal graphs already appears in \cite[Exercise 5.3.22]{West_graph_theory_book}. 
\end{remark}

In addition to the above nice corollary about the value of the independence polynomial at $x=-1$, there has been interest in finding the multiplicity $M(G)$ for cochordal graphs. In particular, $M(G)$ was expressed in terms of the graph structure in \cite{Chordal_multiplicity}. Below, we recall their result and add the relevant implication about the degree of the $h$-polynomial for cochordal graphs.  

\begin{cor}\cite[Corollary 2.8]{Chordal_multiplicity}
 Let $G$ be a cochordal graph where $\overline{G}$ is $r$-connected. Then $M(G)=r$.  As a result, the degree of the $h$-polynomial of the edge ideal of $G$ is 
 $$\deg h_{G} (t)=\alpha(G)-r.$$
\end{cor}

\section{Values of the independence polynomial at $x=-1$ for small $\alpha(G)$}

In this section we study the possible values of \(P_G(-1)\) for connected graphs with \(\alpha(G)\leq 2\). For each \(n\geq 3\), we consider
\[
\mathcal R_n^{(\leq 2)}
:=
\{ P_G(-1): G \text{ is a connected graph on } n \text{ vertices with }
\alpha(G)\leq 2 \}.
\]
We determine this set explicitly, show that every value in it is realized by a connected chordal graph, and then observe that \(|P_G(-1)|\) can already grow exponentially within the connected chordal class.

We start with the following construction.

\begin{defn}\label{defn:chordal_two_clique_family}
Fix integer \(n\geq 3\), \(1\leq a\leq \lfloor n/2\rfloor\), and \(1\leq t\leq n-a\).
Let
\begin{align*}
  A&=\{u_1,\dots,u_a\}\\
B&=\{v_1,\dots,v_{n-a}\}  
\end{align*}
be disjoint sets. Let \(G_n(a,t)\) denote the graph obtained from the disjoint union
of the complete graphs on \(A\) and \(B\) by adding the edges $\{u_1,v_j\}$ for $1\leq j\leq t$. In other words, \(G_n(a,t)\) is obtained from two complete graphs by joining one
distinguished vertex of the first clique to a \(t\)-clique of the second clique.
\end{defn}

\begin{prop}\label{prop:chordal_two_clique_value}
The graph \(G_n(a,t)\) is connected and chordal. Moreover,
\[
P_{G_n(a,t)}(x)=1+nx+\bigl(a(n-a)-t\bigr)x^2.
\]
In particular, \(P_{G_n(a,t)}(-1)=(a-1)(n-a-1)-t\).
\end{prop}

\begin{proof}
Set \(b:=n-a\). First, notice that
\(G_n(a,t)\) is connected since \(u_1\) is adjacent to \(v_1\) as $t\ge 1$.  

Next, we prove that $G_n(a,t)$ is chordal. For this purpose, we show that
\[
u_2,\dots,u_a,\; v_{t+1},\dots,v_b,\; u_1,\; v_1,\dots,v_t
\]
is a perfect elimination ordering. Indeed, each vertex \(u_i\) with \(i\geq 2\)
has all of its remaining neighbors inside the clique \(A\), and each vertex \(v_j\)
with \(j>t\) has all of its remaining neighbors inside the clique \(B\).
After deleting these vertices, the remaining graph on
\(\{u_1,v_1,\dots,v_t\}\) is the clique \(K_{t+1}\). Hence \(G_n(a,t)\) is chordal.

Since both \(A\) and \(B\) are cliques, every independent set of \(G_n(a,t)\)
has size at most \(2\). Thus
\[
P_{G_n(a,t)}(x)=1+nx+g_2x^2,
\]
where \(g_2\) is the number of independent sets of size \(2\). Any such pair must
consist of one vertex of \(A\) and one vertex of \(B\). There are \(ab=a(n-a)\)
such pairs in total, and exactly \(t\) of them, namely
\(\{u_1,v_j\}\) for \(1\leq j\leq t\), are edges. Therefore $g_2=a(n-a)-t$. Hence
\[
P_{G_n(a,t)}(x)=1+nx+\bigl(a(n-a)-t\bigr)x^2.
\]
Therefore,
\[
P_{G_n(a,t)}(-1)=1-n+a(n-a)-t=(a-1)(n-a-1)-t. \qedhere
\]
\end{proof}

\begin{remark}
   Notice that if $\alpha(G)\le 2$, then the complement $\overline{G}$ is triangle-free. There is a well-known bound on the number of edges of such a graph due to Mantel. It states that the maximum number of edges in a triangle-free graph  on $n$
vertices is  $\left\lfloor \frac{n^2}{4}\right\rfloor$.  Equality holds if and only if the graph is $ K_{\lfloor \frac{n}{2}\rfloor,\lceil \frac{n}{2}\rceil}$.
\end{remark}

We obtain the following description utilizing the mentioned result of Mantel from above.

\begin{thm}\label{thm:alpha2_range}
Let \(n\geq 3\). Then
\[
\mathcal R_n^{(\leq 2)}
=
\left[
-(n-1),
\left\lfloor \frac{(n-2)^2}{4}\right\rfloor-1
\right]\cap \mathbb Z
\]
where \(\mathcal R_n^{(\leq 2)}
=
\{ P_G(-1): G \text{ is a connected graph on } n \text{ vertices with } \alpha(G)\leq 2 \}\).
\end{thm}

\begin{proof}
Let \(G\) be a connected graph on \(n\) vertices with \(\alpha(G)\leq 2\). Then
every independent set has size at most \(2\). So
\[
P_G(x)=1+nx+g_2x^2,
\]
where \(g_2\) is the number of independent sets of size \(2\). These independent
pairs are exactly the edges of \(\overline{G}\). Hence $g_2=e(\overline{G})$. Therefore,
\[
P_G(-1)=1-n+e(\overline{G}).
\]

Since \(\alpha(G)\leq 2\), the complement \(\overline{G}\) is triangle-free.
Hence, by Mantel's theorem,
\[
e(\overline{G})\leq \left\lfloor \frac{n^2}{4}\right\rfloor.
\]
Moreover, equality cannot occur because then \(\overline{G}\) would be complete
bipartite, and so \(G\) would be the disjoint union of two cliques, contradicting
the connectedness of \(G\). Thus
\[
e(\overline{G})\leq \left\lfloor \frac{n^2}{4}\right\rfloor-1.
\]
It follows that
\[
P_G(-1)\leq 1-n+\left(\left\lfloor \frac{n^2}{4}\right\rfloor-1\right)
=
\left\lfloor \frac{(n-2)^2}{4}\right\rfloor-1.
\]
On the other hand, \(e(\overline{G})\geq 0\), so
\[
P_G(-1)\geq 1-n=-(n-1).
\]
Hence
\[
\mathcal R_n^{(\leq 2)}
\subseteq
\left[
-(n-1),
\left\lfloor \frac{(n-2)^2}{4}\right\rfloor-1
\right]\cap \mathbb Z.
\]

For the reverse inclusion, fix \(a\) with \(1\leq a\leq \lfloor n/2\rfloor\), and let
\[
I_a:=
\left[
(a-1)(n-a-1)-(n-a), 
(a-1)(n-a-1)-1
\right]\cap \mathbb Z.
\]
By \Cref{prop:chordal_two_clique_value}, as \(t\) runs through
\(1,\dots,n-a\), the graphs \(G_n(a,t)\) realize exactly the integers in \(I_a\).

Notice that  $I_1=[-(n-1),-1]\cap\mathbb Z$. If \(U_a\) denotes the right endpoint of \(I_a\) and \(L_{a+1}\) the left endpoint
of \(I_{a+1}\), then
\[
U_a-L_{a+1}
=
\bigl((a-1)(n-a-1)-1\bigr)
-
\bigl(a(n-a-2)-(n-a-1)\bigr)
=
a-1\geq 0.
\]
Notice that $I_1\cup I_2\cup \cdots \cup I_{\lfloor n/2\rfloor}$ is a single interval of integers since \(I_a\cap I_{a+1}\neq\varnothing\) for every
\(1\leq a<\lfloor n/2\rfloor\).  Then its largest endpoint is attained at \(a=\lfloor n/2\rfloor\) as $(a-1)(n-a-1)$ is increasing for \(1\leq a\leq \lfloor n/2\rfloor\). Therefore
\[
\max (I_1\cup I_2\cup \cdots \cup I_{\lfloor n/2\rfloor}) 
=
\left(\left\lfloor \frac n2\right\rfloor-1\right)
\left(\left\lceil \frac n2\right\rceil-1\right)-1
=
\left\lfloor \frac{(n-2)^2}{4}\right\rfloor-1.
\]
Thus
\[
I_1\cup I_2\cup \cdots \cup I_{\lfloor n/2\rfloor}
=
\left[
-(n-1),
\left\lfloor \frac{(n-2)^2}{4}\right\rfloor-1
\right]\cap \mathbb Z.
\]
Since each \(G_n(a,t)\) is connected, chordal, and satisfies \(\alpha(G_n(a,t))\leq 2\),
the reverse inclusion follows.
\end{proof}

\begin{cor}\label{cor:every_integer_chordal}
Every integer occurs as \(P_G(-1)\) for some connected chordal graph \(G\).
In fact,
\[
\bigcup_{n\geq 3}
\{ P_G(-1): G \text{ connected chordal on } n \text{ vertices with }
\alpha(G)\leq 2 \}
=
\mathbb Z.
\]
\end{cor}

\begin{proof}
By \Cref{thm:alpha2_range}, for each \(n\geq 3\) we obtain the full interval
\[
\left[
-(n-1),
\left\lfloor \frac{(n-2)^2}{4}\right\rfloor-1
\right]\cap \mathbb Z.
\]
As \(n\) varies, the left endpoints tend to \(-\infty\) and the right endpoints
tend to \(+\infty\). Hence every integer is realized.
\end{proof}

Next we  show that $|P_G(-1)|$ can be much larger in the general connected chordal case.

\begin{defn}\label{defn:clique_bouquet}
Let \(q\geq 1\) and let \(r_1,\dots,r_q\) be positive integers. The
\emph{clique bouquet}
\[
H(r_1,\dots,r_q)
\]
is the graph obtained from the cliques \(K_{r_1+1},\dots,K_{r_q+1}\) by
identifying one chosen vertex of each clique to a single common vertex \(x\).
\end{defn}

\begin{prop}\label{prop:clique_bouquet_formula}
The graph \(H(r_1,\dots,r_q)\) is connected and chordal. Moreover,
\[
P_{H(r_1,\dots,r_q)}(x)=x+\prod_{i=1}^q (1+r_i x)
\]
and  \(P_{H(r_1,\dots,r_q)}(-1)=\prod_{i=1}^q (1-r_i)-1\).
\end{prop}

\begin{proof}
The graph is connected since every vertex is adjacent to the common vertex \(x\).
It is chordal because every block is a clique.

An independent set of \(H(r_1,\dots,r_q)\) either contains \(x\), in which case
it contributes the term \(x\), or it avoids \(x\). In the latter case, from each
clique \(K_{r_i+1}\setminus \{x\}\) one may choose either no vertex or exactly one
vertex. Hence the contribution from the \(i\)th clique is \(1+r_i x\). These
choices are independent across \(i\). Therefore
\[
P_{H(r_1,\dots,r_q)}(x)=x+\prod_{i=1}^q (1+r_i x),
\]
and evaluating at \(x=-1\) yields the formula.
\end{proof}

The next theorem shows that the extremal behavior in the full connected chordal
class is exponential rather than quadratic.

\begin{thm}\label{thm:chordal_exponential}
Let
\[
M_n:=\max\{ |P_G(-1)| : G \text{ is a connected chordal graph on } n \text{ vertices} \}.
\]
Then for every \(n\geq 3\),
\[
M_n\geq 4^{\left\lfloor \frac{n-1}{5}\right\rfloor}-1.
\]
In particular, \(M_n\) grows exponentially with \(n\).
\end{thm}

\begin{proof}
For \(n\leq 5\), the stated bound is trivial. Assume \(n\geq 6\), and write $n-1=5q+r$ with  $0\leq r\leq 4$. We construct a clique bouquet on \(n\) vertices as follows:
\[
\begin{cases}
H(\underbrace{5,\dots,5}_{q}) & \text{if } r=0,\\[4pt]
H(\underbrace{5,\dots,5}_{q-1},3,3) & \text{if } r=1,\\[4pt]
H(\underbrace{5,\dots,5}_{q},2) & \text{if } r=2,\\[4pt]
H(\underbrace{5,\dots,5}_{q},3) & \text{if } r=3,\\[4pt]
H(\underbrace{5,\dots,5}_{q},4) & \text{if } r=4.
\end{cases}
\]
In each case the sum of the displayed parameters is \(n-1\), so the graph has \(n\)
vertices. It follows from  \Cref{prop:clique_bouquet_formula} that
\[
P_G(-1)=\prod_i (1-r_i)-1.
\]
Moreover, in the five cases \(r=0,1,2,3,4\), we have
\[
\left|\prod_i (1-r_i)\right|
=
4^q,\ 4^q,\ 4^q,\ 2\cdot 4^q,\ 3\cdot 4^q.
\]
Therefore
\[
|P_G(-1)|
\geq 4^q-1
=
4^{\left\lfloor \frac{n-1}{5}\right\rfloor}-1. \qedhere
\]
\end{proof}

\begin{remark}
It was proved by Engström in \cite[
Corollary 3.2]{Upper_bound_engstrom} that $|P_G(-1)|\leq 2^{\varphi (G)}$ where 
\(\varphi(G)\) denotes the decycling number of \(G\). The decycling number is the  minimum number of vertices whose deletion from $G$ turns it into a forest.

Note that  
\[
4^{\left\lfloor \frac{n-1}{5}\right\rfloor}-1
=2^{ 2\left\lfloor \frac{n-1}{5}\right\rfloor}-1.
\]
Since, for $n\ge 3$,
\[
2\left\lfloor \frac{n-1}{5}\right\rfloor
\le \frac{2n-2}{5}
\le n-2,
\]
it follows that
\[
4^{\left\lfloor \frac{n-1}{5}\right\rfloor}-1
\le 2^{n-2}-1
< 2^{n-2}.
\]

On the other hand, if \(G\) is any connected graph on \(n\ge 2\) vertices, then $\varphi(G)\le n-2$. Indeed, since \(G\) is connected, it contains an edge \(uv\), and deleting all vertices other than \(u\) and \(v\) leaves the graph \(K_2\), which is acyclic. Hence the general bound $|P_G(-1)|\le 2^{\varphi(G)}$
yields $|P_G(-1)|\le 2^{n-2}$.
Therefore, with
\[
M_n:=\max\{ |P_G(-1)| : G \text{ is a connected chordal graph on } n \text{ vertices} \},
\]
we obtain
\[
4^{\left\lfloor \frac{n-1}{5}\right\rfloor}-1
\le M_n
\le 2^{n-2}.
\]

Thus, although the lower bound is far from the  upper bound \(2^{n-2}\), it shows that \(M_n\) already grows exponentially with \(n\) within the class of connected chordal graphs. 
\end{remark}

\bibliographystyle{abbrv}
\bibliography{ref}

\end{document}